\documentclass[12pt]{article}
\usepackage{amsfonts}
\usepackage{amsmath}
\usepackage{amsthm}
\usepackage{amssymb}
\usepackage[mathscr]{eucal}
\usepackage[top=1in,bottom=1in,left=1 in,right=.8in,footskip=0.4in]{geometry}
\usepackage{hyperref}
\usepackage{enumitem}

\newcommand{\IND}{\mathop{\mathrm{ind}}}
\newcommand{\CORE}{\mathop{\mathrm{core}}}
\newcommand{\KER}{\mathop{\mathrm{ker}}}
\newcommand{\DIADEM}{\mathop{\mathrm{diadem}}}
\newcommand{\DEG}{\mathop{\mathrm{deg}}}
\newcommand{\D}{d}
\newcommand{\DC}{d_c}
\newcommand{\J}{\mathscr I}

\newtheorem{THEOREM}{Theorem}[section]
\newtheorem{LEMMA}[THEOREM]{Lemma}

\newtheorem{COROLLARY}[THEOREM]{Corollary}
\newtheorem{PROBLEM}[THEOREM]{Problem}
\newtheorem{CONJECTURE}[THEOREM]{Conjecture}
\newtheorem{PROCEDURE}[THEOREM]{Procedure}
\newtheorem{NOTE}[THEOREM]{Note}

\usepackage{tikz}
\usepackage{calc}
\usepackage{float}
\usetikzlibrary{calc}
\tikzstyle{vertex}=[circle, draw, fill=black, inner sep=0pt, minimum size=6pt]
\tikzstyle{redv}=[circle, draw=red,very thick, inner sep=0pt, minimum size=6pt]
\tikzstyle{bluv}=[circle, draw=blue, inner sep=3.5pt, minimum size=6pt]
\tikzstyle{bluvin}=[circle, draw=blue, fill=blue, inner sep=0pt, minimum size=6pt]

\newcommand{\vertex}{\node[vertex]}
\newcommand{\redv}{\node[redv]}
\newcommand{\bluv}{\node[bluv]}
\newcommand{\bluvin}{\node[bluvin]}

\title{\bf 
Problems on Matchings and Independent Sets of a Graph
}

\author{Amitava Bhattacharya, Anupam Mondal\\
\small School of Mathematics\\[-0.8ex]
\small Tata Institute of Fundamental Research\\[-0.8ex] 
\small Mumbai, India\\
\small\tt \{amitava, anupamm\} @\,math.tifr.res.in\\
\and
T. Srinivasa Murthy\\
\small Department of Computer Science and Automation\\[-0.8ex]
\small Indian Institute of Science\\[-0.8ex] 
\small Bangalore, India\\
\small\tt srinivasat @\,iisc.ac.in\\
}

\begin{document}
\maketitle

\begin{abstract} \noindent
Let $G$ be a finite simple graph. For $X \subset V(G)$, the \emph{difference} of $X$, $\D (X) := |X| - |N (X)|$ where
$N(X)$ is the neighborhood of $X$ and $\max \, \{\D(X):X\subset V(G)\}$  is called the
\emph{critical difference} of $G$. 
$X$ is called a \emph{critical set} if $\D (X)$ equals the critical difference 
and $\KER (G)$ is the intersection of all critical sets. It is known that $\KER (G)$ is an independent (vertex) set of $G$.  $\DIADEM (G)$ 
is the union of all critical independent sets. An independent set $S$ is an \emph{inclusion minimal set with $\D(S) > 0$} if no proper subset 
of $S$ has positive difference. 

A graph $G$ is called \emph{K\"{o}nig-Egerv\'{a}ry} if the sum of its independence number ($\alpha (G)$) and matching number 
($\mu (G)$) equals the order of $G$. It is known that bipartite graphs are K\"{o}nig-Egerv\'{a}ry.

In this paper, we study independent sets with positive difference for which every proper subset has a smaller difference and prove a 
result conjectured by Levit and Mandrescu in 2013. The conjecture states that for any graph, the number of inclusion minimal sets $S$ with $\D(S) > 0$ 
is at least the critical difference of the graph. We also give a short proof of the inequality 
$|\KER (G)| + |\DIADEM (G)| \le 2\alpha (G)$ (proved by Short in 2016). 

A characterization of unicyclic non-K\"{o}nig-Egerv\'{a}ry graphs is also presented and a conjecture which states that for 
such a graph $G$, the critical difference equals $\alpha (G) - \mu (G)$, is proved.

We also make an observation about $\KER (G)$ using Edmonds-Gallai Structure Theorem as a concluding remark.
 \end{abstract}

\section{Introduction}
In this paper  $G$ is a finite simple graph with the vertex set $V(G)$ and the edge
set $E(G) \subset \binom{V(G)}{2}$. 
For $A\subset V(G)$, neighborhood of $A$, denoted by $N_G(A)$, is the set of all vertices adjacent to some
vertex in $A$. When there is no confusion the subscript $G$ may be dropped. The \emph{degree} of a vertex $v$ is the number of edges incident to $v$ and is denoted by $\DEG(v)$. If $U \subset V(G)$, then $G[U]$ is the graph with the vertex set $U$ and whose edges are precisely the edges of $G$ with both ends in $U$. A set $S$ of vertices is \emph{independent} if no two vertices from $S$ are
adjacent. The cardinality of a largest independent set (\emph {maximum independent set}) is denoted by $\alpha(G)$.
The reader may refer to any of the standard text books~\cite{diestel,lovaszplummer}
for basic notations.

This paper is based on the results in~\cite{2015arXiv150600255J, larson1, 2002_1, 2012_3, 2012_1,2013_2, 2013_1, 2013_3, short, 1990_1}. We state the following definitions which will help us to formulate the results proved in this paper.

Let $\IND(G) = \{S:S \mbox{ is an independent set of } G\}$,
$\Omega(G)=\{S\in \IND (G):|S|=\alpha (G)\}$ and
${\CORE}(G)=\cap\{S:S\in\Omega(G)\}$\cite{2002_1}. 

For $X\subset V(G)$, the number $|X| -|N(X)|$ is called the \emph{difference} of the set $X$ and is denoted by $\D_G (X)$. Again we drop the subscript $G$ in case of no ambiguity. The number $\max \, \{\D(X):X\subset V(G)\}$ is called the \emph{critical difference} of $G$ and is denoted by $\DC(G)$. A set $U\subset V(G)$
is \emph{critical} if $\D(U)=\DC(G)$ \cite{1990_1} and $\KER (G)$ is the intersection of all critical sets~\cite{2012_1}.
\emph{Diadem} of a graph is defined as the union of all critical independent sets and it is denoted by
$\DIADEM (G)$ 
\cite{2015arXiv150600255J}. One may observe that $\KER (G) \subset \CORE (G)$ \cite{2012_1} and thus $\KER (G) \in \IND (G)$.
An independent set $S$ is an inclusion minimal set with $\D(S) > 0$ 
if no proper subset of $S$ has positive difference \cite{2013_2}. 

A matching in a graph $G$ is a set $M$ of edges 
such that no two edges in $M$ share a
common vertex. 
Size of a largest possible matching (\emph {maximum matching}) is denoted by $\mu(G)$. A vertex is \emph{matched} 
(or \emph{saturated}) by $M$ if it is an endpoint of one of the edges in $M$. A perfect matching is a matching which 
matches all vertices of the graph. For two disjoint subsets $A$ and $B$ of $V(G)$ we say 
there is a \emph {matching from $A$ into $B$} if there is a matching $M$ such that any edge in $M$ joins a vertex in $A$ and a vertex in $B$ 
and all the vertices in $A$ are matched by $M$.

A graph $G$ is a K\"{o}nig-Egerv\'{a}ry (KE) graph
if $\alpha(G) + \mu(G) = |V(G)|$ \cite{dem, ster}. K\"{o}nig-Egerv\'{a}ry graphs have been well studied.
Levit and Mandrescu studied the critical difference, $\KER$, $\CORE$, $\DIADEM$ of graphs, properties of K\"{o}nig-Egerv\'{a}ry 
graphs and proved several results. Based on these results several natural conjectures and problems arose. The ones considered in 
this paper are stated below.

\begin{CONJECTURE}\cite{2013_2}
	\label{1.1}
	For any graph $G$, the number of inclusion minimal independent set $S$ such that $\D(S) > 0$ is at least $\DC(G)$.
\end{CONJECTURE}

\begin{THEOREM}
	\label{1.2} For any graph $G$,
	$|\KER (G)| + |\DIADEM (G)| \le 2\alpha(G )$. 
	\footnote{Conjectured in~\cite{2015arXiv150600255J} and proved in~\cite{short}.}
\end{THEOREM}

\begin{CONJECTURE}
	\label{1.3}
	For a unicyclic non-KE graph $G$, $\DC(G)=\alpha (G)-\mu (G)$. 
	\footnote {Stated by Levit in a talk at Tata Institute of Fundamental Research in 2014.}
\end{CONJECTURE}

\begin{PROBLEM}\cite{2015arXiv150600255J, 2013_1}
	\label{1.4}
	Characterize graphs such that $\CORE (G)$ is critical.
\end{PROBLEM}

\begin{PROBLEM}\cite{2015arXiv150600255J, 2013_1}
	\label{1.5}
	Characterize graphs with $\KER (G) = \CORE (G)$.
\end{PROBLEM}

In Section~2, Conjecture~\ref{1.1} and related results are proved. In Section~3 a new short proof of Theorem~\ref{1.2} is given.
A characterization of unicyclic non-K\"{o}nig-Egerv\'{a}ry graph is presented in Section~4 and as a corollary
Conjecture~\ref{1.3} is deduced. In the concluding section Edmonds-Gallai structure Theorem is used to make an
observation regarding $\KER (G)$. It may be useful for Problems~\ref{1.4} and
\ref{1.5}.

\section{On Minimum Number of Inclusion Minimal Sets with Positive Difference}
In this section we study $X \in \IND (G)$ with $\D (X)>0$ such that for all $Y \subsetneq X$, $\D(Y)<\D(X)$ and give a proof of Conjecture~\ref{1.1}.
There are several results that led to the formulation of this conjecture.  Some of them are listed 
below as they help to understand the proof  or they are used in the proof of the conjecture.

\begin{THEOREM} \cite{larson1} \label{2.1}
	There is a matching from $N(S)$ into $S$ for every critical independent
	set $S$.
\end{THEOREM}

\begin{THEOREM}\cite{2012_1}\label{2.2}
	For every graph $G$, the following assertions are true:
	\begin{enumerate}[label=(\roman*)]
		\item $\KER (G)$ is the unique minimal critical independent set of $G$.
		\item $\KER (G) \subset \CORE(G)$.
	\end{enumerate}
\end{THEOREM}

\begin{THEOREM}\cite{2012_1}
	\label{2.3}
	For a graph $G$, the following assertions are true:
	\begin{enumerate}[label=(\roman*)]
		\item The function $\D$ is supermodular, i.e., $\D(X \cup Y) + \D(X \cap Y) \ge \D(X) + \D(Y)$ for
		every $X, Y \subset V (G)$.
		\item If $X$ and $Y$ are critical in $G$, then $X \cup Y$ and $X \cap Y$ are critical as well.
	\end{enumerate}
\end{THEOREM}

\begin{THEOREM} \cite{2013_1}
	\label{2.4}
	If $G$ is a bipartite graph, then $\KER (G) = \CORE (G)$.
\end{THEOREM}

\begin{THEOREM} \cite{2013_2}
	\label{2.5}
	For a vertex $v$ in a graph $G$, the following assertions hold:
	\begin{enumerate}[label=(\roman*)]
		\item $\DC(G - v) = \DC(G) - 1$ if and only if $v \in \KER (G)$ ;
		\item if $v \in \KER (G)$, then $\KER (G - v) \subset \KER (G) - v $.
	\end{enumerate}
\end{THEOREM}

\begin{THEOREM}\cite{2013_2}
	\label{2.6}
	If $\KER (G) \neq \emptyset$, then\\
	$\KER(G) = \cup\{ S : S$ is an inclusion minimal independent set with $\D(S) > 0 \}$
\end{THEOREM}

For an independent set $X$ of $G$  a new graph $H_X$ is defined as follows. The vertex set $V(H_X) = X \cup N(X) \cup \{v, w\}$, where $v$ and $w$
are two new vertices not in $V(G)$ and the edge set
$E(H_X) = \{ xy \in E(G): x \in X, y \in N(X)\} \cup \{vw\}\cup \{vx : x \in N(X)\}$. Note that if $G$ is a connected graph with $|V(G)|>1$, then $H_X$ is a connected bipartite graph. Figure~\ref{fig1} gives an illustration of the construction.
Also observe that for all $Y\subset X$, $\D_{H_X}(Y)= \D_G(Y)$.
\begin{figure}[!ht]
	\[\begin{tikzpicture}
	\node at (1,1.5) {$G:$};
	\vertex (a) at (2,0) {};
	\vertex (b) at (3,0) {};
	\vertex (c) at (4,0) {};
	\vertex (d) at (5,0) {};
	
	\vertex (p) at (2.5,3) {};
	\vertex (q) at (3.5,3) {};
	\vertex (r) at (4.5,3) {};
	
	\vertex (x) at (2.5,5) {};
	\vertex (y) at (3.5,5) {};
	\vertex (z) at (4.5,5) {};
	
	\draw [gray,thick] (3.5,0) ellipse (2cm and .4cm);
	\node at ($(3.5,0)+(-90:2 and .8)$) {$X$};
	
	\draw [gray,thick] (3.5,3) ellipse (1.5cm and .3cm);
	\node at ($(3.5,3)+(90:2 and .6)$) {$N(X)$};
	
	\path [line width=1pt]
	
	(a) edge (p)
	(b) edge (p)
	(c) edge (q)
	(d) edge (q)
	(d) edge (r)
	(q) edge (r)
	(p) edge (x)
	(p) edge (y)
	(x) edge (y)
	(y) edge (z)
	(z) edge (r)
	;  
	
	\node at (7,1.5) {$H_X:$};
	\vertex (a1) at (8,0) {};
	\vertex (b1) at (9,0) {};
	\vertex (c1) at (10,0) {};
	\vertex (d1) at (11,0) {};
	
	\vertex (p1) at (8.5,3) {};
	\vertex (q1) at (9.5,3) {};
	\vertex (r1) at (10.5,3) {};
	
	\redv (v1) at (13,0) [label=right:$v$] {};
	\redv (w1) at (13,3) [label=right:$w$] {};
	
	\draw [gray,thick] (9.5,0) ellipse (2cm and .4cm);
	\node at ($(9.5,0)+(-90:2 and .8)$) {$X$};
	
	\draw [gray,thick] (9.5,3) ellipse (1.5cm and .3cm);
	\node at ($(9.5,3)+(90:2 and .6)$) {$N(X)$};
	
	\path [line width=1pt]
	
	(a1) edge (p1)
	(b1) edge (p1)
	(c1) edge (q1)
	(d1) edge (q1)
	(d1) edge (r1)
	
	(v1) edge (p1)
	(v1) edge (q1)
	(v1) edge (r1)
	(v1) edge (w1)
	;   
	\end{tikzpicture}\]
	\caption{ Construction of $H_X$ for a given independent set $X$ of $G$.}
	\label{fig1}
\end{figure}
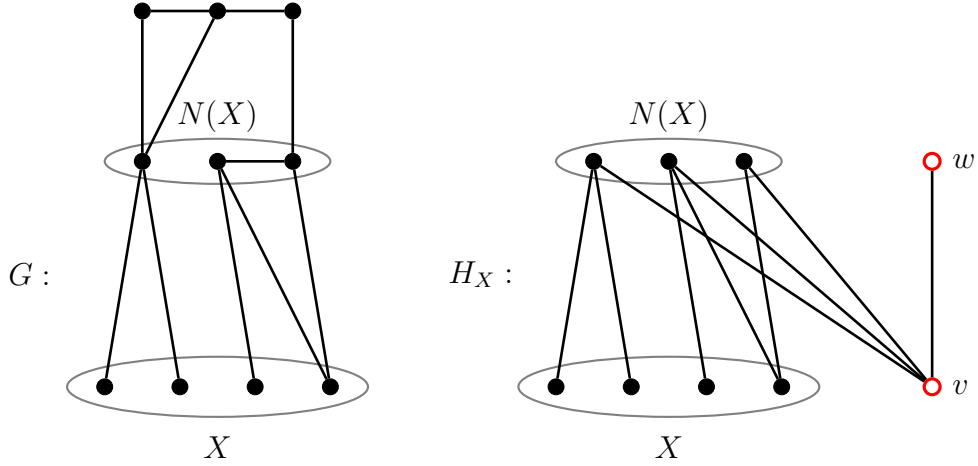

\begin{THEOREM} \label{2.7}
	If $X$ is an independent set of $G$ with $\D(X) > 0$ such that for all  $Y \subsetneq X$, $\D(Y) < \D(X)$, then $\KER(H_X) = X$.
\end{THEOREM}

\begin{proof}
	Note that the only maximal independent set that 
	contains $v$ is $X\cup \{v\}$ and its size is $|X| +1$. 
	Also if a maximal independent set does not contain $v$, then it contains $w$. Now we show that size of a 
	largest independent set that contains $w$ is also $|X|+1$.
	
	Let $I = S \cup T \cup \{w\}$ be a largest independent set where $S \subset X$ and $T \subset N(X)$.
	\begin{align*}
	& & &|N(S) \cup T|& & = & &|N(X)| \ \mbox{(since $I$ is a maximal independent set, $N(S)\cup T=N(X)$)}& \\
	&\Rightarrow& &|N(S)| +|T|& & = & &|N(X)| \ \mbox{(since $N(S)$ and $T$ are disjoint)}& \\
	&\Rightarrow& &|N(S)|+ |X-S|& &\le & &|N(X)| \ \mbox{(otherwise $|X\cup \{w\}|> |I|$)}& \\
	&\Rightarrow& &|X|-|N(X)|& &\le & & |S| -N(S)| & \\
	& \Rightarrow& & \D(X) & &\le & & \D(S)  \mbox{, a contradiction if $S\subsetneq X$.}&
	\end{align*}

	Thus $I= X\cup \{w\}$. 
	Hence $X \cup \{v\}$ and $X \cup \{w\}$ are the only two largest independent sets of $H_X$. 
	Thus $\CORE (H_X)= (X\cup \{v\}) \cap (X \cup \{w\})= X$.
	Since $H_X$ is bipartite, Theorem~\ref{2.4} implies  
	$\KER (H_X) = \CORE (H_X) = X$.
\end{proof}

\begin{COROLLARY}
	\label{2.8}
	If $X$ is an independent set of $G$ with $\D(X) > 0$ and for all $Y \subsetneq X$, $\D(Y) < \D(X)$, then $X$ can be expressed as a union of inclusion minimal independent sets with positive difference.
\end{COROLLARY}
\begin{proof}
	By Theorem~\ref{2.6}, $X=\KER (H_X)$ is the union of all inclusion minimal independent sets of $H_X$ with positive difference. 
	As inclusion minimal independent sets of $H_X$ with positive difference are contained in $X$, they are inclusion minimal independent sets of $G$ with positive difference. Hence the result follows.
\end{proof}

\begin{COROLLARY}
	If $X$ is an independent set of $G$ with $\D(X) > 0$ such that for all $Y \subsetneq X$, $\D(Y) < \D(X)$, then $X \subset 
	\KER(G)$.
\end{COROLLARY}
\begin{proof} From Corollary~\ref{2.8} it follows that $X$ is contained in the union of inclusion minimal independent sets of $G$ with positive difference.
	From Theorem~\ref{2.6} it follows that $X$ is contained in $\KER (G)$.
\end{proof}

\begin{NOTE}\label{2.10}
	\normalfont 
	A set $S$ with $\D(S) > 0$ such that no proper subset of $S$ has positive difference must be an independent set. If not, we have $S \cap N(S) \ne \emptyset$. Let $X = S \cap N(S)$. Since $N(S - X) \cap S = \emptyset$, we have $N(S - X) \subset N(S) - S = N(S) - X$. Thus 
	\begin{align*}
	\D(S - X) &= |S - X| - |N(S - X)|\\
	&\ge |S - X| - |N(S) - X|\\
	&= |S| - |X| - (|N(S)| - |X|)\\
	&= \D(S) > 0,
	\end{align*}
	a contradiction. Thus we may drop the word ``\emph{independent}" from the phrase ``$S$ is an inclusion minimal independent set with $\D(S)>0$".
\end{NOTE}

\begin{COROLLARY}\label{2.11}
	If $S$ is an inclusion minimal set of a graph $G$ with $\D_G(S)>0$, then $\D_G(S) = 1$.
\end{COROLLARY}
\begin{proof} This result was first shown in~\cite{2013_2}. From Note~\ref{2.10} it follows that $S$ is an independent set. Since $S$ is an independent set with $\D_G(S)>0$ and no proper subset of $S$ has positive difference, from Theorem~\ref{2.7} we have $\KER (H_S) = S$. Hence by Theorem~\ref{2.5}(i), for all $v \in S$, 
	\begin{align*}
	\D_{H_S} (S - \{v\}) &= \DC(H_S) - 1\\
	&= \D_{H_S} (S)-1 \quad \mbox{(by Theorem~\ref{2.2}(i), $S = \KER (H_S)$ is critical in $H_S$)}\\
	&= \D_G (S)-1.
	\end{align*}
	Thus for all $v \in S$, $\D_G(S - \{v\}) =\D_G(S) - 1$. As $S$ is an inclusion minimal set with $\D_G(S) > 0$, we have 
	$\D_G(S - \{v\}) \le 0$, which implies 
	$\D_G(S) \le 1$. Hence $\D_G(S) = 1$.
\end{proof}

Corollary~\ref{2.8}  can be made stronger and it can also be proved directly.

\begin{THEOREM} \label{2.12}
	Let $X \in \IND (G)$ with $\D(X) = k > 0$. If for all $Y \subsetneq X$, $\D(Y) < k$, then $X$ can be expressed as a union of 
	$k$ distinct inclusion minimal sets with positive difference.
\end{THEOREM}

\begin{proof}
	We note that if $ S \subset X$  then 
	\begin{align*}
	\D(X - S) & =  |X-S| - |N(X-S)| \\
	& =  |X| -|S| -   |N(X-S)| \\
	&\ge  |X|- |N(X)| -|S|\\
	&=  \D(X)- |S|\\
	&=  k - |S|.
	\end{align*} 
	
	We note that if $\D(S)>0$, then $S$ contains an inclusion minimal set $T$ with $\D(T)>0$. As the poset $\left( \{ X \subset S : \D(X) > 0 \}, \subset \right)$ is nonempty and finite, it admits a minimal element which can be chosen as $T$.
	
	Since $\D(X) > 0$, $X$ contains an inclusion 
	minimal set with positive difference. Choose an inclusion minimal set $S_1\subset X$ with positive difference and a vertex $s_1\in S_1$.  Suppose  $(s_1, S_1), \ldots, (s_i,S_i)$, where $i<k$  have been selected.
	As $\D(X-\{s_1,\ldots, s_i\})\ge k-i>0$, $X-\{s_1, \ldots, s_i\}$  contains an inclusion minimal set with positive difference. 
	Choose an inclusion minimal set $S_{i+1} \subset X-\{s_1,\ldots, s_i\}$ with positive difference and a vertex $s_{i+1} \in S_{i+1}$.  
	This process is continued  till the  pairs  $(s_1,S_1), \ldots , (s_k,S_k)$ are obtained. Note that for integers $i$ and $j$ with $1 \le i < j \le k$, we have $s_i \notin S_{j}$. Also from Corollary~\ref{2.11} it follows that for $i \in \{ 1, \ldots, k\}$, $\D (S_i) = 1$.
	
	Next we show that  $\D(S_1 \cup \cdots \cup S_k) \ge k$ by repeated application of 
	supermodularity of $\D$ (Theorem~\ref{2.3}). 
	
	By supermodularity of $\D$, for $i\in \{1,\ldots , k-1\}$, 
	\begin{align}\label{2.12e1}
	&\D(S_k\cup S_{k-1}\cup \cdots \cup S_i ) + \D\left( (S_k \cup S_{k-1}\cup \cdots \cup S_{i+1})\cap  S_i \right) \ge \nonumber \\
	&\D(S_k \cup S_{k-1}\cup \cdots \cup S_{i+1}) + \D(S_{i}).
	\end{align}
	
	Now $s_{k-1} \in S_{k-1} - S_k$ implies $S_k \cap S_{k-1} \subsetneq S_{k-1}$. Since $S_{k-1}$ is inclusion minimal with $\D(S_{k-1}) > 0$ and $S_k \cap S_{k-1} \subsetneq S_{k-1}$, we have 
	$\D(S_k \cap S_{k-1}) \le 0$.
	Also $\D(S_{k-1}) = \D(S_k) = 1$.  Thus equation~\eqref{2.12e1} for $ i=k-1$ implies  $\D(S_k \cup S_{k-1}) \ge 2$.

	In general, for $i\in \{1,\ldots , k-1\}$, $s_i\in S_i -\cup _{j=i+1}^kS_j$. Hence 
	$(S_k \cup S_{k-1}\cup \cdots \cup S_{i+1})\cap  S_i \subsetneq S_i$,
	implying $\D((S_k \cup S_{k-1}\cup \cdots \cup S_{i+1})\cap  S_i )\le 0$. Therefore, if $\D(S_k \cup S_{k-1}\cup \cdots \cup S_{i+1} ) \ge k-i$, then equation~\eqref{2.12e1} implies 
	$\D(S_k \cup S_{k-1}\cup \cdots \cup S_{i} ) \ge k-i+1$.
	In particular $\D (S_k \cup S_{k-1}\cup \cdots \cup S_{1} ) \ge k$. This leads to a contradiction if
	$\cup_{i=1}^k S_i \subsetneq X$. Thus $X= \cup_{i=1}^k S_i $.
\end{proof}

\begin{COROLLARY}
	The number of inclusion minimal independent sets of $G$ with positive difference is greater than or equal to $\DC(G)$.
\end{COROLLARY}

\begin{proof} 
	As $\KER(G)$ is the unique minimal critical independent set of $G$, setting $X=\KER (G)$ in 
	Theorem~\ref{2.12} this corollary and Conjecture~\ref{1.1} is proved.
\end{proof}

Converse of Corollary~\ref{2.8} also holds true.
\begin{THEOREM}
	If $X \subset V(G)$ can be expressed as a union of inclusion minimal sets with positive difference, then 
	\begin{enumerate}[label=(\roman*)]
		\item $X \in \IND(G)$;
		\item $\D (X) > 0$ ;
		\item for all $Y \subsetneq X$, $\D (Y) < \D(X)$.
	\end{enumerate}
	
\end{THEOREM}
\begin{proof}
	Let $X = \cup_{i=1}^k S_i$ where $S_i$ is an inclusion minimal set with $\D (S_i) > 0$ for all $i \in \{1, \dots, k\}$. 
	Since $X$ is contained in the union of inclusion minimal sets with positive difference, $X \subset \KER(G)$ and thus $X \in \IND(G)$. This proves (i).
	
	First note that if $k=1$, then $X=S_1$. In this case (ii) and (iii) hold, so assume $k \ge 2$. If $Y \subsetneq X$, then at least one of the sets $S_1, \ldots , S_k$ is not contained in $Y$, say $S_1$. 
	Observe that for any $Z \subset X$ and $i \in \{1, \dots, k-1\}$, we have $\left(Z \cup S_1 \cup \cdots \cup S_{i}\right)\cap  S_{i+1}  \subset S_{i+1}$. This implies $\D\left((Z \cup S_1 \cup \cdots \cup S_{i})\cap  S_{i+1} \right) \le \D(S_{i+1})$ for all $i \in \{1, \dots, k-1\}$. Now by supermodularity of $\D$, for $i \in \{1, \dots, k-1\}$,
	\begin{align}
	\D(Z \cup S_1\cup \cdots \cup S_{i+1} )  & \ge 
	\D(Z \cup S_1 \cup\cdots \cup S_{i}) + \D(S_{i+1}) - \D\left((Z \cup S_1 \cup \cdots \cup S_{i})\cap  S_{i+1} \right) \nonumber \\ 
	& \ge \D(Z \cup S_1 \cup \cdots \cup S_{i}) \label{2.14e2}
	\end{align}
	By repeated application of equation~\eqref{2.14e2}, we get 
	\begin{align}\label{2.14e3}
	\D(X) = \D(Z \cup S_1 \cup \cdots \cup S_{k}) \ge \D(Z \cup S_1)
	\end{align}
	Setting $Z = \emptyset$ in equation~\eqref{2.14e3}, we get $\D(X) \ge \D(S_1)>0.$ This proves (ii).\\
	Setting $Z = Y$ in equation~\eqref{2.14e3}, we get 
	\begin{align}\label{2.14e4}
	\D(X) \ge \D(Y \cup S_1)
	\end{align}
	By supermodularity of $\D$, 
	\begin{align}\label{2.14e5}
	\D(Y \cup S_1) + \D(Y \cap S_1) \ge \D(Y)+\D(S_1)
	\end{align}
	As $S_1 \not\subset Y$, we have $Y \cap S_1 \subsetneq S_1$ and thus $\D(Y \cap S_1) < \D(S_1)$. Hence from  equation~\eqref{2.14e4} and equation~\eqref{2.14e5} we get 
	$$\D(X) \ge \D(Y \cup S_1) \ge \D(Y)+ \D (S_1) - \D(Y \cap S_1) > \D(Y).$$ 
	This proves (iii).
\end{proof}

We state one more result on criticality of independent sets which will be used is Section~4. This is a converse of Theorem~\ref{2.1}.
\begin{THEOREM} \label{2.15}
	Let $X$ be an independent set of a graph $G$ containing $\KER (G)$. If there is a matching from $N(X)$ into $X$, then $X$ is critical. 
\end{THEOREM}

\begin{proof}
	Let $A = X - \KER (G)$ and 
	$B = N(X) - N(\KER (G))$. 
	Since there is a matching from $N(X)$ into $X$ and there are no edges from $B$ to $\KER (G)$, Hall's
	Marriage Theorem~\cite{diestel, lovaszplummer} implies that $$|B| \le |N(B) \cap X| \le |X - \KER (G)| = |A|.$$
	Thus 
	\begin{align*}
	\D(X) &= |X| - |N(X)| \\    
	&= (|\KER (G)| + |A|) - (|N(\KER (G))| + |B|)\\
	&= \DC(G) + |A| - |B| \\
	&\ge \DC(G).
	\end{align*}
	Hence $\D(X) = \DC(G)$.
\end{proof} 

\begin{COROLLARY}
	$\CORE (G)$ is critical if and only if there is a matching from $N(\CORE (G))$ into $\CORE (G)$. 
\end{COROLLARY}

\begin{proof}
	Follows from Theorem~\ref{2.1}, Theorem~\ref{2.2} and Theorem~\ref{2.15}.
\end{proof} 
\section{A Ker-Diadem Inequality}

It was conjectured in \cite{2015arXiv150600255J} that the sum of sizes of $\KER$ and $\DIADEM$ of a graph is at most twice
the independence number of the graph. Short proved this inequality in \cite{short} 
using structural results by Larson in \cite{larson1}. Here a short and direct proof for this ``$\KER$-$\DIADEM$ inequality"  is presented.

\begin{LEMMA} \label{3.1}
	If $X$ and $Y$ are two critical independent sets of a graph $G$ then $|N (X) \cap Y| = |N (Y) \cap X|$.
\end{LEMMA}

\begin{proof}
	Let $S = N(X) \cap Y$. As $S \subset Y, N(S) \subset N(Y)$, and hence $N(S) \cap X \subset N(Y) \cap X$. Using Hall's
	Marriage Theorem it can be verified that if $X$ is a critical independent set then there is a matching from 
	$N(X)$ into $X$ (Theorem~\ref{2.1}).
	As $S \subset N(X)$, Hall's Marriage Theorem implies $|S| \le |N(S) \cap X| \le |N(Y) \cap X|$,
	i.e., $|N(X) \cap Y| \le |N(Y) \cap X|$.
	By a similar argument, we have $|N(Y) \cap X| \le |N(X) \cap Y|$.
\end{proof}

\begin{THEOREM}\label{3.2}
	If $X$ is a maximal critical independent set of a graph $G$ then $\DIADEM (G) \subset X \cup N(X) - N(\KER (G))$.
\end{THEOREM}

\begin{proof}
	First we show that $\DIADEM (G) \subset X \cup N(X)$. Toward a contradiction suppose $x \in \DIADEM (G) - (X \cup N(X))$. 
	As $x \in \DIADEM (G)$, there exists a critical independent set $A$ containing $x$. 
	Let $Y = X \cup A - (N(X) \cap A)$. Observe that  $Y$ is an independent set and $X \subsetneq Y$. This implies
	$N(Y) \subset N(X \cup A) - (X \cap N(A))$ (in fact equality holds). Hence
	\begin{align*}
	\D(Y) &= |Y| - |N(Y)| \\
	& \ge  |X \cup A - (N(X) \cap A)| - |N(X \cup A) - (X \cap N(A))| \\
	& =   |X \cup A| - |N(X) \cap A| - |N(X \cup A)| + |X \cap N(A)| \\
	&= |X \cup A| - |N(X \cup A)| \mbox{ (using Lemma~\ref{3.1}) }\\
	&= \D(X \cup A).
	\end{align*}
	By supermodularity  of $\D$ and criticality of $X$ and $A$, it follows that $X \cup A$ is also critical (Theorem~\ref{2.3}). Thus $\D(Y) = \DC(G)$
	and this 
	contradicts that $X$ is a maximal critical independent set. This shows that 
	$\DIADEM (G) \subset X \cup N(X)$. 
	
	Now for any  critical independent set $I$,  
	$I \cap N(\KER (G)) = \emptyset$. Thus taking union over all such $I$, we get $\DIADEM (G) \cap N(\KER (G)) = \emptyset$. Hence $\DIADEM (G) \subset X \cup N(X) - N(\KER(G))$.
\end{proof}

\begin{COROLLARY}
	For every graph $G$, $|\DIADEM (G)| + |\KER (G)| \le 2\alpha (G)$.
\end{COROLLARY}

\begin{proof}
	Let $X$ be a maximal critical independent set of $G$. By Theorem~\ref{3.2}, 
	$\DIADEM (G) \subset X \cup N(X) - N(\KER (G))$.
	Therefore, 
	\begin{align*}
	|\DIADEM (G)| &\le |X \cup N(X)| - |N(\KER(G))| \quad \mbox{ (as $N(\KER (G))  \subset N(X))$ }\\
	&= |X| + |N(X)| - |N(\KER(G))| \quad \mbox{ (as $X$ is independent) }\\
	&= |X| + (|X| - \DC(G)) - |N(\KER(G))| \quad \mbox{ (as $X$ is critical) }\\
	&= 2|X| - (\DC(G) + |N(\KER(G))|) \\
	&= 2|X| - |\KER(G)| \\
	&\le 2\alpha(G) - |\KER(G)|.
	\end{align*}
\end{proof}

\section{Characterization of Unicyclic non-KE Graphs}

For any graph $G$, $\alpha(G)+\mu(G) \le |V(G)|$ and for bipartite graphs equality holds \cite{dem,ster}. It turned out that many interesting 
properties can be proved for a graph $G$ which satisfy $\alpha(G)+\mu(G) = |V(G)|$. This motivated the study of graphs for which this equality holds.
A graph is called  K\"{o}nig-Egerv\'{a}ry (KE)  if $ \alpha(G)+\mu(G) = |V(G)|$. KE graphs have been studied extensively~\cite{2015arXiv150600255J}.
This motivated researchers to consider graphs that are ``close" to KE graphs. One of these classes is unicyclic graphs. For any 
unicyclic graph $G$, $|V(G)|-1 \le \alpha(G)+\mu(G) \le |V(G)|$.

In this section we characterize the unicyclic non-KE graphs and prove some properties. These properties also 
lead to a proof of the Conjecture~\ref{1.3}.

We present a procedure to construct any connected non-KE graph $G$.
\begin{PROCEDURE}
	\label{4.1}
	A connected graph $G$ is constructed by the following Steps.
	\begin{enumerate}
		\item Construct an odd cycle, color its vertices blue.
		\item Go to Step (3) or Step (4) or  stop.
		\item  Attach a path of length two to any vertex:  Choose a vertex $u$ and then add two new vertices $u_1$, $u_2$ and 
		two edges $uu_1$, $u_1u_2$. Color $u_1$ red and $u_2$ black. Go to Step (2).
		\item Attach a black leaf to a red vertex: If there is no vertex colored red, go to Step (2). Else choose a red vertex $u$ and add a new vertex $u_1$ and an edge $uu_1$. 
		Color $u_1$ black and go to Step (2).
	\end{enumerate}
\end{PROCEDURE}

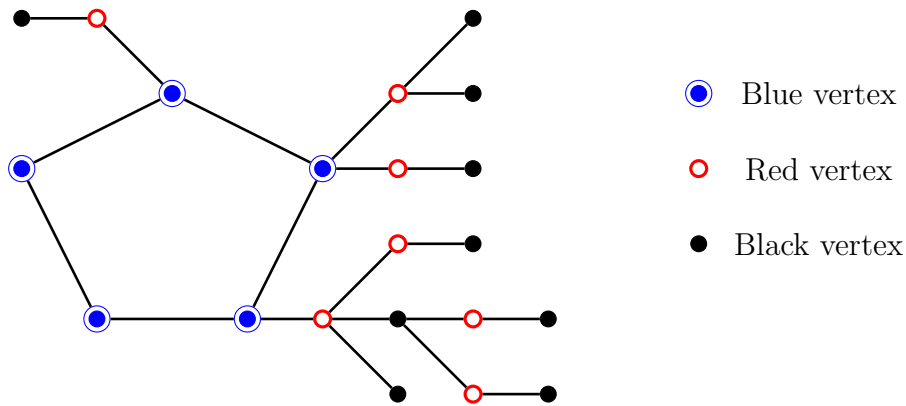
\begin{figure}[!ht]
	\[\begin{tikzpicture} 
	\bluv (a) at (2,0) {};
	\bluv (b) at (4,0) {};
	\bluv (c) at (5,2) {};
	\bluv (d) at (3,3) {};
	\bluv (e) at (1,2) {};
	\bluvin () at (2,0) {};
	\bluvin () at (4,0) {};
	\bluvin () at (5,2) {};
	\bluvin () at (3,3) {};
	\bluvin () at (1,2) {};
	
	\redv (r1) at (6,3) {};
	\vertex (b11) at (7,3) {};
	\vertex (b12) at (7,4) {};
	
	\redv (r2) at (5,0) {};
	\vertex (b21) at (6,0) {};
	\vertex (b22) at (6,-1) {};
	
	\redv (r3) at (6,1) {};
	\vertex (b31) at (7,1) {};
	
	\redv (r4) at (7,0) {};
	\vertex (b41) at (8,0) {};
	
	\redv (r5) at (6,2) {};
	\vertex (b51) at (7,2) {};
	
	\redv (r6) at (7,-1) {};
	\vertex (b61) at (8,-1) {};
	
	\redv (r7) at (2,4) {};
	\vertex (b71) at (1,4) {};
	
	\bluv () at (10,3) {};
	\node at ($(10,3)+(0:1.6)$) {Blue vertex};
	\bluvin () at (10,3) {};
	\redv () at (10,2) {};
	\node at ($(10,2)+(0:1.6)$) {Red vertex};
	\vertex () at (10,1) {};
	\node at ($(10,1)+(0:1.6)$) {Black vertex};

	\path [line width=1pt]
	
	(a) edge (b)
	(b) edge (c)
	(c) edge (d)
	(d) edge (e)
	(e) edge (a)
	
	(b) edge (r2)
	(c) edge (r1)
	(c) edge (r5)
	(d) edge (r7)
	
	(r1) edge (b11)
	(r1) edge (b12)
	(r2) edge (b21)
	(r2) edge (b22)
	(r3) edge (b31)
	(r4) edge (b41)
	(r5) edge (b51)
	(r6) edge (b61)
	(r7) edge (b71)
	(b21) edge (r4)
	(b21) edge (r6)
	(r2) edge (r3)
	
	;  
	\end{tikzpicture}\]
	\caption{ A graph constructed by Procedure~\ref{4.1}.}
	\label{fig2}
\end{figure}

Let $G$ be a graph constructed by Procedure~\ref{4.1}. Henceforth the unique cycle in $G$ will be denoted by $C$. 
$G-E(C)$ is a forest. Consider the components of this forest 
as rooted trees with the root being the blue vertex of $C$. The notions of \emph{parent}, \emph{child}, \emph{descendant} etc. 
are used with respect to these rooted trees. Note that leaves are always colored black. See 
Figure~\ref{fig2} for an example of such graph.

\begin{LEMMA} \label{4.2}
	If a graph $G$ constructed by Procedure~\ref{4.1} is unicyclic non-KE, then the graph $G'$ obtained from $G$ after applying Step~3 once is also unicyclic non-KE.
\end{LEMMA}

\begin{proof} 
	We may add the new edge $u_1u_2$ to any matching $M$ in $G$ and the vertex $u_2$ to any independent set $I$ of $G$ to get a new matching
	$M'$ in $G'$ and an independent set $I'$ of $G'$. This implies   $\mu(G')+\alpha(G')\ge \mu(G)+\alpha(G)+2$ 
	
	Since the added two extra vertices  are adjacent, $\alpha (G')\le \alpha (G)+1$. Also the two added edges have a vertex in common. 
	Thus $\mu (G')\le \mu (G)+1$. This yields $\mu(G')+\alpha(G')\le \mu(G)+\alpha(G)+2$. 
	Therefore,  $\mu(G')+\alpha(G') = |V(G')|-1$.
\end{proof}

\begin{THEOREM}
	\label{4.3}
	There exists a largest independent set that contains all the black vertices of a graph $G$ constructed by 
	Procedure~\ref{4.1}.
\end{THEOREM}

\begin{proof}
	Note that black vertices are added one at a time in Step~(3) or Step~(4) of Procedure~\ref{4.1}. 
	We use induction on the number of black  vertices to prove this theorem.\\
	\emph{Base Case:} If there are no black vertices the result is trivially true.\\
	\emph{Induction Hypothesis:} There exists a largest independent set that contains all the black vertices when the number of black vertices is $\ell-1$.\\
	\emph{Induction Step:} Number of black vertices is $\ell$. Look at the most recently added
	black vertex in $G$ by Step~(3) or Step~(4). Call it $w$ and its parent (which is a red vertex) $v$.\\
	\indent \emph {Case 1:} $w$ is added by Step (3).\\
	The graph $G' = G - v - w$ can be constructed by Procedure~\ref{4.1} and the number of black vertices in $G'$ is $\ell-1$. By induction hypothesis, 
	there exists a largest independent set $I'$ of $G'$ containing all the black vertices. Note that $I' \cup \{w\}$ is a largest independent set of $G$ and it contains all the black vertices of $G$.\\
	\indent \emph {Case 2:} $w$ is added by Step~(4).\\
	The graph $G' = G - w$ can be constructed by Procedure~\ref{4.1} and the number of black vertices in $G'$ is $\ell-1$. By induction hypothesis, 
	there exists a largest independent set $I'$ of $G'$ containing all the black vertices. 
	The vertex $v$, which is the parent $w$, has at least one black child other than $w$. 
	This implies $v\notin I'$. Hence $I' \cup \{w\}$ 
	is a largest independent of $G$ and it contains all the 
	black vertices of $G$.
\end{proof}

\begin{THEOREM}
	Every largest matching in a graph $G$ constructed by 
	Procedure~\ref{4.1} covers all the red vertices.
\end{THEOREM}

\begin{proof}
	We prove this by induction on the number of red vertices.\\
	\emph{Base Case:} Result is true when the number of red vertices is zero.\\
	\emph{Induction Hypothesis:} Every largest matching covers all the red vertices when the number of 
	red vertices is $\ell-1$.\\
	\emph{Induction Step:} Number of red vertices is $\ell$. \\
	Let the most recently added red vertex be $v$. 
	The only descendants of $v$ are black leaves, $w_1, ... , w_k$, where $k \ge 1$. 
	Let $M$ be a largest matching in $G$. It may be verified that 
	$G' = G - \{ v, w_1, ... , w_k\}$ can be obtained by Procedure~\ref{4.1} (it can be realized as a subsequence of steps that produced $G$) and the number of red 
	vertices in $G'$ is $\ell-1$.
	
	Let $M' = M \cap E(G')$. Observe that  $|M'| \ge |M|-1$. 
	Now $M'$ is a largest matching in $G'$, otherwise there exists a matching $M''$ in $G'$ with $|M''| > |M'|$. 
	But $M'' \cup \{vw_1\}$ is a matching in $G$ larger than $M$, a contradiction. 
	By induction hypothesis, $M'$covers all the $\ell-1$ red vertices of $G'$. 
	Now $M'$ cannot be a largest matching of $G$ as $M' \cup \{vw_1\}$ is a larger matching 
	in $G$. Thus $|M| - |M'| = 1$ and hence $M$ covers all the red vertices in $G$.
\end{proof}

\begin{COROLLARY}\label{4.5} 
	If a graph $G$ constructed by Procedure~\ref{4.1} is unicyclic non-KE, then the graph $G'$ obtained from $G$ after applying Step~(4) once is also unicyclic non-KE.
\end{COROLLARY}

\begin{proof}
	Let $G'$ be the graph obtained from $G$ after an application of Step (4) with a new black leaf $u_1$ attached to a red vertex $u$ of $G$. 
	Since $u$ is matched by 
	all largest matchings in $G$, $\mu(G') = \mu (G)$. Also by Theorem~\ref{4.3}, $\alpha(G')\le \alpha(G)+1$ (in fact equality holds here). 
	Thus $G'$ is unicyclic non-KE. 
	
\end{proof}

\begin{THEOREM}
	A graph obtained by Procedure~\ref{4.1} is a connected unicyclic non-KE graph.
\end{THEOREM}
\begin{proof} Follows from Lemma~\ref{4.2} and Corollary~\ref{4.5} and the fact that an odd cycle is a non-KE graph. 
\end{proof}

Now we shall show that every connected unicyclic non-KE graph $G$ can be obtained by Procedure~\ref{4.1}. The
notation $C$ is continued to denote the unique cycle in $G$. If $C$ is an even cycle, then $G$ is bipartite and thus KE. Hence $C$ must be an odd cycle. Consider the forest $F$ obtained from $G$ by deleting
the edges in the cycle $C$. Define the vertices in the cycle as the roots of the trees in this forest.  
Let $T_v$ be the component of the forest $F$ rooted at $v$. 
The root $v$ is not considered a leaf 
even if the degree of $v$ in $T_v$ is 1. $T_v$ is called nontrivial if it has more than one vertex.

\begin{LEMMA}\label{4.7}
	If $G$ is a unicyclic non-KE graph with the cycle $C$, then $G$ does not have a leaf attached to $C$.
\end{LEMMA}

\begin{proof}
	Suppose to the contrary that there exists a leaf $v$ attached to a vertex $w$ of $C$.
	If $G' = G - v- w$, then $G'$ is a forest (hence KE) 
	and thus $\alpha(G') + \mu(G') = |V(G')| = |V(G)| - 2$.
	
	Now if $I$ is a largest independent set of $G'$, then $I \cup \{v\}$ is an independent set of $G$, and if $M$ is a largest matching in $G'$, 
	then $M \cup \{vw\}$ is a 
	matching in $G$. 
	Thus $\alpha(G) \ge |I \cup \{v\}| = |I| + 1 = \alpha(G') + 1$, 
	and $\mu(G) \ge |M \cup \{vw\}| = |M| + 1 = \mu(G') + 1$.
	Hence $\alpha(G) + \mu(G) \ge \alpha(G') + \mu(G') + 2 = |V(G)|$, implying $G$ is a KE graph, a contradiction.
\end{proof}

\begin{LEMMA}\label{4.8}
	Let $G$ be a connected unicyclic non-KE graph with the (unique) cycle $C$ and $F$ be the forest obtained from $G$ by deleting all the edges belonging to $C$. If $T_v$ is the component of $F$ rooted at $v \in E(C)$, then every nontrivial $T_v$ contains a non-root vertex $x$  with one of the following properties:
	\begin{enumerate}[label=(\roman*)]
		\item $x$ is the parent of more than one leaves.
		\item $x$ is the parent of only one leaf and the degree of $x$ is $2$.
	\end{enumerate}
\end{LEMMA}

\begin{proof}
	Suppose $T_v$ does not contain any non-root vertex with property (i). We assert that $T_v$ must contain a vertex with property (ii). 
	Suppose not, then each leaf has a parent of degree more than 2 and the parent does not have another leaf as a child. 
	Let $L$  be the set of  leaves of $T_v$. From  Lemma~\ref{4.7} it follows that $v \notin N(L)$. 
	Let $P = V(T_v) - (L \cup N(L) \cup \{v\})$. 
	Clearly $V(T_v) = L \sqcup N(L) \sqcup P \sqcup \{v\}$ ($\sqcup$ denotes disjoint union of sets). 
	It follows from the assumption that $|N(L)| = |L|$ and for all $x \in N(L), \DEG(x) \ge 3$. 
	Also $\DEG(v) \ge 1$, for all $z \in P$, $\DEG(z) \ge 2$, and for all $y \in L$, $\DEG(y) = 1$. 
	Hence $2|E(T_v)| =$ sum of the vertex degrees in $T_v \ge$ $$|L| + 3|N(L)| + 2|P| + 1 = 4|L| + 2|P| + 1.$$ 
	Thus $|E(T_v)| > 2|L| + |P|$.
	But $|V(T_v)| = |L| + |N(L)| + |P| + 1$ implies $$|E(T_v)| = |L| + |N(L)| + |P| = 2|L| + |P|, \mbox{ a contradiction.}$$ 
\end{proof}

\begin{THEOREM}\label{4.9}
	Any connected unicyclic non-KE graph $G$ can be constructed by Procedure~\ref{4.1}. 
\end{THEOREM}

\begin{proof}
	Let $G$ be a minimal connected unicyclic non-KE graph that cannot be constructed by Procedure~\ref{4.1}. By Lemma~\ref{4.8}, it follows that $G$ contains a nontrivial $T_v$ and a vertex $x$ of $T_v$ with one of the two properties 
	listed in the lemma. \\
	\indent \emph {Case 1:} $x$ satisfies property (i).\\
	Let  $x$ be the parent of the leaves $y_1, \ldots, y_k$, where $k \ge 2$. 
	Note that all the leaves $y_1, \ldots, y_k$ belong to any largest independent set of $G$. Let $G' = G - y_k$. 
	Note that $\alpha(G') = \alpha(G) - 1$ and  $\mu(G') \le \mu(G)$ (in fact equality holds). This implies 
	$$\alpha(G') + \mu(G') \le \alpha(G) - 1 + \mu(G) < |V(G)| - 1 = |V(G')|.$$ 
	Thus $G'$ is a connected unicyclic non-KE subgraph of $G$. Minimality of $G$ implies $G'$ can be constructed by Procedure~\ref{4.1}.
	But $G$ can be obtained from $G'$ by applying Step (4) and hence $G$ can also be constructed by Procedure~\ref{4.1}, a contradiction.
	\\
	\indent \emph {Case 2:} $x$ satisfies property (ii).\\ 
	Let the unique child of $x$ be the leaf $y$ and $G' = G - x- y$. 
	Note that if $I$ is a largest independent set of $G$, then either $x \in I$ or $y \in I$ (but not both). Also if $z$ is the parent of $x$ and $M$ is a largest matching in $G$, then either $zx \in M$ or $xy \in M$, but not both. Thus $\alpha(G') \le \alpha(G) - 1$ and $\mu(G') \le \mu(G) - 1$. This implies
	$$\alpha(G') + \mu(G') \le \alpha(G) - 1 + \mu(G) - 1 < |V(G)| - 2 = |V(G')|.$$ 
	Thus $G'$ is a connected unicyclic non-KE subgraph of $G$. Minimality of $G$ implies $G'$ can be constructed by Procedure~\ref{4.1}.
	But $G$ can be obtained from $G'$ by applying Step~(3) and hence $G$ can also be constructed by Procedure~\ref{4.1}, a contradiction.
\end{proof}

\begin{THEOREM}
	For any connected unicyclic non-KE graph $G$ the vertex coloring generated by Procedure~\ref{4.1} is independent of the 
	particular sequence of steps that results in $G$. 
\end{THEOREM}
\noindent Proof of this theorem is omitted. Reduction steps similar to the ones used in the proof of Theorem~\ref{4.9} 
may be used here as for each reduction step the choice of color(s) for the deleted vertex (vertices) is unique. It may be noted 
that though the 
coloring is unique the same graph $G$ can be generated by different sequence of steps. 
\bigskip\\
Let $G$ be a connected unicyclic non-KE graph with the odd cycle $C$. For the rest of this section we assume $C$ is of length $2m+1$. Also we assume the unique coloring of the vertices of $G$ induced by Procedure~\ref{4.1}. 
Define $B$ and $R$ to be the set of black 
and red vertices of $G$ respectively.

\begin{THEOREM}\label{4.11}
	For any connected unicyclic non-KE graph $G$, $\CORE (G) \subset B$ and $\alpha(G) = |B| + m$.
\end{THEOREM}

\begin{proof} 
	By Theorem~\ref{4.3} there exists a largest independent set $I$ of $G$ containing $B$, so there is no red vertex in $I$. The set $I' := I - B \subset V(C)$  is a largest independent set of the graph $C$. Thus $|I'| = m$.  For any $x\in V(C)$, there exists a largest independent set $I_x$ of $C$ such that $x \notin I_x$. 
	Note that $B \cup I_x$ is a largest 
	independent of $G$.  Hence for any $x\in V(C)$, $x \notin \CORE (G)$. Thus $V(C) \cap \CORE (G) = \emptyset$ and
	$\CORE (G) \subset  B$. Also $\alpha(G) =|I|=|B\cup I'|=|B|+m$.
\end{proof} 

An edge is called \emph{red-black} if one endpoint of the edge is a red vertex and the other endpoint is 
a black vertex. 

\begin{LEMMA}
	\label{4.12}
	There is a largest matching $M$ in any connected unicyclic non-KE graph $G$ such that $R$ is covered by only red-black edges.
\end{LEMMA}

\begin{proof} 
	We use induction on $|R|$ to prove this lemma.\\
	\emph{Base Case:} $|R| = 0$, where it is trivially true.\\
	\emph{Induction Hypothesis:} Let the assertion be true for $|R| = \ell - 1$. \\
	\emph{Induction Step:}
	Let $|R|=\ell$ and $v$ be the last red vertex  added in $G$  by Procedure~\ref{4.1}. 
	Note that the only descendants of $v$ are black 
	leaves $w_1, ... , w_k$, where $k \ge 1$. Now $G' = G - \{v, w_1, ... , w_k\}$ is also a unicyclic non-KE 
	graph with $\ell-1$ red vertices. By induction hypothesis, $G'$ has a largest matching $M$ which covers all the red vertices by only red-black edges. 
	Note that $M \cup \{vw_1\}$ 
	is a largest matching in $G$. Since $vw_1$ is a red-black edge, the lemma is proved.
\end{proof} 

\begin{COROLLARY}\label{4.13}
	For any connected unicyclic non-KE graph $G$, $\mu (G) = |R| + m$.
\end{COROLLARY}

\begin{proof}
	Let $M$ be a largest matching in $G$ that covers all the red vertices by only red-black edges. 
	If $e \in M$ is not a red-black edge, then $e \in E(C)$. 
	Thus number of non-red-black edges in $M$ is the size of a largest matching in the graph $C$, which is $m$. 
	Hence $\mu(G) = |R| + m$.
\end{proof}

\begin{THEOREM} \label{4.14}
	For any connected unicyclic non-KE graph $G$, $B$ is a critical set of $G$. 
\end{THEOREM}

\begin{proof}
	Note that $\KER (G) \subset \CORE (G)\subset B $ (Theorem~\ref{2.2} and Theorem~\ref{4.11}). 
	Observe that in Procedure~\ref{4.1} whenever a red vertex is added, an adjacent black vertex is also added. Thus $R \subset N(B)$. 
	Also a black vertex is never attached to another black vertex or a blue vertex (vertex belonging to the unique cycle $C$). Thus $N(B) = R$. Also by Lemma~\ref{4.12}, 
	there is a matching from $R$ into $B$. Hence by Theorem~\ref{2.15}, $B$ is critical.
\end{proof} 

\begin{COROLLARY}
	For any connected unicyclic non-KE graph $G$, $\DC(G) = |B| - |R| = \alpha(G) - \mu(G)$.
\end{COROLLARY}

\begin{proof}
	Follows from Theorem~\ref{4.11}, Corollary~\ref{4.13} and Theorem~\ref{4.14}.
\end{proof}
If $G$ is a disconnected unicyclic non-KE graph, then $G=G' \oplus F$, 
($\oplus$ denotes disjoint union of graphs), where $G'$ 
is the component of $G$ containing the unique (odd) cycle and $F$ is a nonempty forest. 
Observe that $G'$ is a connected unicyclic non-KE graph. 
Conversely, if $G'$ is any connected unicyclic non-KE graph and $F$ is an arbitrary forest, then 
$G=G' \oplus F$ is a unicyclic non-KE graph.

By the corollary above, $\DC(G') = \alpha(G') - \mu(G')$. It is known that for KE graphs 
the critical difference equals the difference between independence number and matching number \cite{2012_3}. 
Hence $\DC(F)=\alpha(F)-\mu(F)$. 
It may be verified that $\DC(G) = \DC(G')+\DC(F)$, $\alpha(G) = \alpha(G')+\alpha(F)$ and
$\mu(G)= \mu(G')+\mu(F)$. 
Thus $\DC(G) = \alpha(G') - \mu(G') + \alpha(F)-\mu(F)= \alpha(G) - \mu(G)$.
Hence $\DC(G) = \alpha(G) - \mu(G)$ holds for any disconnected unicyclic non-KE graph $G$ too and thus Conjecture~\ref{1.3} is proved.
\bigskip\\

\begin{COROLLARY}
	Let $G$ be a connected unicyclic non-KE graph. Consider any sequence of Steps in Procedure~\ref{4.1} that results in $G$. Then the critical difference of $G$ is  the number of times a black vertex is added by Step (4) in the sequence.
\end{COROLLARY}

\begin{proof}
	As the red vertices in the construction of a graph $G$ are in a one-to-one correspondence with the black vertices added by Step (3),
	the result follows. 
\end{proof}

\section{Concluding Observations}
Core, ker, diadem of a graph and other related notions have been well studied, but some basic questions still remain. Two of the least understood problems are: 

\begin{PROBLEM}\cite{2015arXiv150600255J, 2013_1}
	Characterize graphs such that $\CORE (G)$ is critical.
\end{PROBLEM}

\begin{PROBLEM}\cite{2015arXiv150600255J, 2013_1}
	Characterize graphs with $\KER  (G)=\CORE (G)$.
\end{PROBLEM}

We tried to use Edmonds-Gallai decomposition  to understand $\KER (G)$ better. The observations in this section seem to be insufficient
to address the above problems but may be useful in further work. For the completeness and to fix the notation  Edmonds-Gallai decomposition is stated below. 

Let $G$ be any finite simple graph. Let $\mathcal{D}$ be the set of vertices not covered (\emph{missed}) by a largest matching 
in $G$ and $\mathcal{A}$ be the set of neighbors of $\mathcal {D}$ outside $\mathcal D$.
The set $\mathcal C$ contains the remaining vertices. Note that $V(G) = \mathcal{A} \sqcup \mathcal{C} \sqcup \mathcal{D}$.
Edmonds-Gallai decomposition of $G$ is the partition of $V(G)$
into the three sets $\mathcal{A}, \mathcal{C}$ and $\mathcal{D}$.

\begin{THEOREM}[Edmonds-Gallai Structure Theorem]\cite{lovaszplummer}
	\label{5.3}
	Let $\mathcal{A}$, $\mathcal{C}$ and $\mathcal{D}$  be the sets in the Edmonds-Gallai  
	decomposition of a graph $G$. Let $G_1$, \ldots, $G_k$ be the components of $G[\mathcal{D}]$.
	If $M$ is a largest matching in $G$, then the following properties hold.
	\begin{enumerate}[label=(\roman*)]
		\item All the vertices in $\mathcal {C}$ are matched amongst themselves by $M$ (which implies $G[\mathcal{C}]$ has a perfect matching).
		\item If  $S$ is a nonempty subset of $\mathcal{A}$, then $N(S)$ has a vertex in at least $|S| + 1$ distinct components 
		of $G[\mathcal{D}]$.
		\item All the vertices in $\mathcal{A}$ are matched with vertices belonging to distinct components of $G[\mathcal{D}]$ by $M$. In other words, for any pair of vertices $x$ and $y$ belonging to $\mathcal{A}$, there are distinct integers $i$ and $j$ in $\{1, \ldots, k\}$ such that $x$ and $y$ are matched with a vertex belonging to $V(G_i)$ and a vertex belonging to $V(G_j)$ respectively by $M$.
		\item Each of $G_1$, \ldots, $G_k$ is ``factor-critical" (a graph $H$ is called factor-critical if for all $v \in V(H)$, the graph $H - v$ has a perfect matching).
	\end{enumerate}
\end{THEOREM}

First a simple observation is stated.

\begin{LEMMA}
	\label{5.4}
	Let $G$ be a disjoint union of factor critical graphs each of order strictly greater than 1. If $S \in \IND(G)$ and $S \ne \emptyset$, then $\D_G (S) < 0$.
\end{LEMMA}

\begin{proof}
	Without loss of generality it may be assumed that $G$ is connected. Since $G$ is connected and $|V(G)| > 1$, each vertex of $G$ has at least one neighbor. This implies (as $S \ne \emptyset$), $N_G (S) \ne \emptyset$. Now choose a vertex $v \in N_G (S)$.
	Let $G' = G - v$. Since $G$ is a factor critical graph, $G'$ has a perfect matching. 
	As $S$ is an independent set of $G'$ and $G'$ admits a perfect matching, $|N_{G'}(S)| \ge |S|$. Observe that $N_{G'}(S) = N_G (S) - \{v\}$. Thus $|N_G (S) - \{v\}| \ge |S|$. Therefore, $\D_G (S) =|S|- |N_G (S)|= |S|- |N_G (S) - \{v\}| -1 < 0$.
\end{proof}

\begin{THEOREM}
	\label{5.5}
	Let $S$ be a critical independent set of $G$. If $G_1, \ldots , G_k$ are the components of $G[\mathcal{D}]$, then 
	$$S 
	\subset \mathcal{C} \cup \left(\bigsqcup\limits _{
		\substack{ i = 1 \\ |V(G_i)| = 1}
	}^k
	V(G_i)\right).$$
\end{THEOREM}

\begin{proof}
	Let 
	$$Y=\mathcal{C} \cup \left(\bigsqcup\limits _{
		\substack{ i = 1 \\ |V(G_i)| = 1}
	}^k
	V(G_i)\right)$$ 
	and
	$$X =  S\cap (V(G) - Y) = S \cap \left(\mathcal{A} \cup \left(\bigsqcup\limits _{
		\substack{ i = 1 \\ |V(G_i)| > 1}
	}^k
	V(G_i)\right)\right).$$ 
	We need to show that $X = \emptyset$.
	Suppose to the contrary that $X \neq \emptyset$.\\ 
	\emph{Assertion:} $|X| - |N(X) \cap \mathcal{D}| < 0$.\\
	Let $\J=\{ i \in \{ 1, \ldots, k\} : |V(G_i)|>1 \mbox{ and } V(G_i)\cap S \neq \emptyset \}$ and $m=|\J|$.
	
	Let $X_i = S  \cap V(G_i)$, for $i\in \J$ and $X_0 = S  \cap \mathcal{A}$. 
	Note that $X =X_0\sqcup (\sqcup_{i \in \J } X_i)$.
	Since for all $i \in \{ 1, \ldots, k\}$, $G_i$ is factor critical, Lemma~\ref{5.4} implies that for $ i \in \J$, $|N(X_i) \cap V(G_i)| \ge |X_i|+1$. Thus 
	$$\sum_{i \in \J} |N(X_i) \cap V(G_i)| \ge \sum_{i\in \J} |X_i| + m.$$
	
	\indent \emph {Case 1:} $X_0 = \emptyset$.\\
	Since in this case $m> 0$,
	\begin{align*}
	&  \sum_{i \in \J} |N(X_i) \cap V(G_i)| > \sum_{i \in \J} |X_i| = |X|\\
	\Rightarrow & |N(X) \cap \mathcal{D}| = 
	|(\cup_{i \in \J} N(X_i)) \cap  \mathcal{D}|  > |X|.
	\end{align*}
	\indent \emph{Case 2:} $X_0 \neq \emptyset$.\\
	Let $W = 
	\sqcup_{i \in \J} V(G_i)$.
	Now we shall show that 
	$|N(X_0) \cap (\mathcal{D} - W)| > |X_0| - m$ when $|X_0|\ge m$.
	By Theorem~\ref{5.3}(ii), $X_0$ has neighbors in at least $|X_0| +1$ 
	components of $G[\mathcal{D}]$.
	Thus $X_0$ has neighbors in at least $|X_0| + 1 - m$  components of $G[\mathcal{D}]$ different from $G_i$,
	for $i\in \J$. 
	Hence $|N(X_0) \cap (\mathcal{D} - W)| \ge |X_0| + 1 - m > |X_0| - m$.  Also
	$$|(\cup_{i \in \J} N(X_i)) \cap W | = \sum_{i \in \J} |N(X_i) \cap V(G_i)| \ge 
	\sum_{i\in \J} |X_i| + m.$$
	Therefore, 
	\begin{align*}
	|N(X) \cap \mathcal{D}| &\ge |N(\cup_{i \in \J} X_i) \cap W| + |N(X_0) \cap (\mathcal{D} - W)| \\
	&>  \sum_{i \in \J} |X_i| + m + |X_0| - m \\
	&= |X|.
	\end{align*}
	Thus in both the cases $|N(X) \cap \mathcal{D}| > |X|$ and the assertion is proved. \\\\
	Next note that 
	$$S - X \subset Y = \mathcal{C} \cup \left(\bigsqcup\limits _{
		\substack{ i = 1 \\ |V(G_i)| = 1}
	}^k
	V(G_i)\right).$$
	In other words, all the vertices in $S - X$ belong to either the set $\mathcal{C}$ or the singleton components of $G[\mathcal{D}]$. Now vertices belonging to the singleton components of $G[\mathcal{D}]$ do not have neighbors in any component of $G[\mathcal{D}]$. Also it follows from Edmonds-Gallai decomposition that there are no edges joining vertices in $\mathcal{D}$ and vertices in $\mathcal{C}$. Thus $N(Y) \cap \mathcal{D} = \emptyset$ and in particular $N(S - X) \cap \mathcal{D} = \emptyset$.
	This implies 
	\begin{align*}
	N(S - X) &\subset N(S) - \mathcal{D} \\
	&\subset N(S) - (N(S) \cap \mathcal{D}).
	\end{align*}
	Thus $|N(S - X)| \le |N(S)| - |N(S) \cap \mathcal{D}| = |N(S)| - |N(X) \cap \mathcal{D}|$.
	From these observations it follows that 
	\begin{align*}
	\D(S - X) &= |S| - |X| - |N(S - X)| \\
	&\ge  |S| - |X| - |N(S)| + |N(X) \cap \mathcal{D}| \\
	&= \D(S) + |N(X) \cap \mathcal{D}| - |X| \\
	&> \D(S).
	\end{align*}
	This contradicts  the criticality of $S$.
\end{proof}

\begin{COROLLARY}
	If $G_1, G_2, ... , G_k$ are the components of $G[\mathcal{D}]$, then 
	$$\KER (G) \subset \bigsqcup\limits _{
		\substack{ i = 1 \\ |V(G_i)| =1}
	}^k
	V(G_i).$$
\end{COROLLARY}

\begin{proof}
	Since $\KER (G)$ is a critical independent set of $G$, Theorem~\ref{5.5} implies that 
	$$\KER (G) \subset \mathcal{C} \cup \left(\bigsqcup\limits _{
		\substack{ i = 1 \\ |V(G_i)| =1}
	}^k
	V(G_i)\right).$$
	Let $X =\KER(G) \cap \mathcal{C}$.  We shall show that $X = \emptyset$. 
	Suppose to the contrary that $X \ne \emptyset$.
	By Theorem~\ref{5.3}(i), vertices in $\mathcal {C}$ are matched amongst themselves by any largest matching. 
	Thus there is a matching from $X$ into $N(X) \cap \mathcal{C}$, which implies $|N(X) \cap \mathcal{C}| \ge |X|$. 
	Now observe that 
	$$\KER (G) - X \subset \bigsqcup\limits _{
		\substack{ i = 1 \\ |V(G_i)| =1}
	}^k
	V(G_i).$$
	In other words, $\KER (G) - X$ is a set of vertices belonging to the singleton components of $G[\mathcal{D}]$. It follows from Edmonds-Gallai decomposition that there are no edges joining vertices in $\mathcal{D}$ and vertices in $\mathcal{C}$. Thus vertices belonging to the singleton components of $G[\mathcal{D}]$ have neighbors (if any) only in $\mathcal{A}$. In particular, $N(\KER(G) - X) \subset \mathcal{A}$, and hence 
	\begin{align*}
	N(\KER(G) - X) &\subset N(\KER(G)) \cap \mathcal{A} \\
	&\subset N(\KER(G)) -  \mathcal{C} \\
	&\subset N(\KER(G)) - (N(X) \cap \mathcal{C}).
	\end{align*}
	Therefore, 
	\begin{align*}
	\D(\KER(G) - X) &= |\KER(G)| - |X| - |N(\KER(G) - X)| \\
	&\ge |\KER(G)| - |X| - |N(\KER(G))| + |N(X) \cap \mathcal{C}| \\
	&\ge \DC(G),
	\end{align*}
	a contradiction to the minimality of $\KER(G)$ as a critical independent set.
\end{proof}

It would also be nice to generalize the results proved for unicyclic non-KE graphs to graphs $G$ for which 
$\alpha (G)+ \mu (G) = |V(G)|-k$, where $k$ is a constant. It would be interesting to look at properties 
of $\KER$, $\CORE$, $\DIADEM$ etc. for graphs that are ``close" to bipartite graphs. 

\section{Acknowledgement}
Amitava Bhattacharya would like to thank  Vadim Levit for introducing this topic to him during his short visit to 
TIFR, Mumbai in 2014.  His initial insights and slides helped us start this project. His continued support has been very helpful.  
Srinivasa Murthy would like to express deep gratitude to  his thesis advisor  S. M. Hegde for his support and making this 
collaboration possible, which started during his thesis work. Srinivasa Murthy also thanks The 
Institute of Mathematical Sciences, Chennai, India, along with National Board for Higher Mathematics, India, 
for the post doctoral fellowship.

 \bibliographystyle{plain}

\begin{thebibliography}{9}

\bibitem{dem}
R.~W.~Deming.
\newblock Independence numbers of graphs-- an extension of the K{\"o}nig--Egerv{\'a}ry theorem.
\newblock {\em Discrete Mathematics}, 27: 23-33, 1979. 
 
\bibitem{diestel}
R.~Diestel.
\newblock Graph theory, fifth edition.
\newblock {\em Springer-Verlag, New York}, 2010.

\bibitem{2015arXiv150600255J}
A.~Jarden, V.~E.~Levit, and E.~Mandrescu.
\newblock Critical and maximum independent sets of a graph.
\newblock {\em arXiv:1506.00255}, page~12, 2015.

\bibitem{larson1}
C.~E.~Larson.
\newblock A note on critical independence reductions.
\newblock {\em Bull. Inst. Comb. Appl.}, 5:34–46, 2007.

\bibitem{2002_1}
V.~E.~Levit and E.~Mandrescu.
\newblock Combinatorial properties of the family of maximum stable sets of a
  graph.
\newblock {\em Discrete Applied Mathematics}, 117(1):149--161, 2002.

\bibitem{2012_3}
V.~E.~Levit and E.~Mandrescu.
\newblock {Critical independent sets and K{\"o}nig--Egerv{\'a}ry graphs}.
\newblock {\em Graphs and Combinatorics}, 2012.

\bibitem{2012_1}
V.~E.~Levit and E.~Mandrescu.
\newblock Vertices belonging to all critical sets of a graph.
\newblock {\em SIAM Journal on Discrete Mathematics}, 26(1):399--403, 2012.

\bibitem{2013_2}
V.~E.~Levit and E.~Mandrescu.
\newblock On the structure of the minimum critical independent set of a graph.
\newblock {\em Discrete Mathematics}, 313(5):605--610, 2013.

\bibitem{2013_1}
V.~E.~Levit and E.~Mandrescu.
\newblock Critical sets in bipartite graphs.
\newblock {\em Annals of Combinatorics}, 17(3):543--548, 2013.

\bibitem{2013_3}
V.~E.~Levit and E.~Mandrescu.
\newblock On the intersection of all critical sets of a unicyclic graph.
\newblock {\em Discrete Applied Mathematics}, 162: 409--414, 2014.

\bibitem{lovaszplummer}
L.~Lov{\'a}sz and M.~D.~Plummer.
\newblock {Matching theory}, volume~29 of {\em North-Holland Mathematics
  Studies 121, Annals of Discrete Mathematics}.
\newblock North-Holland, 1986.

\bibitem{short}
T.~Short.
\newblock On some conjectures concerning critical independent sets of a graph.
\newblock {\em The Electronic Journal of Combinatorics}, 23(2):P2.43, 2016.

\bibitem{ster}
F.~Stersoul.
\newblock A characterization of the graphs in which the transversal number equals the matching number.
\newblock {\em Journal of Combinatorial Theory}, B27:228-229, 1979.

\bibitem{1990_1}
C.~Zhang.
\newblock Finding critical independent sets and critical vertex subsets are
  polynomial problems.
\newblock {\em SIAM Journal on Discrete Mathematics}, 3(3):431--438, 1990.

\end{thebibliography}

\end{document}